\documentclass[11pt,reqno]{amsart}
\usepackage{graphicx}
\usepackage{wrapfig}
\usepackage{fancyhdr}
\usepackage{amsmath,amssymb,latexsym,verbatim}
\newtheorem{teo}{Theorem}

\newtheorem{lem}[teo]{Lemma}

\begin{document}

\title{Inscribed and circumscribed polygons that characterize inner product spaces}

\author[C. Ben\'itez]{Carlos Ben\'itez}
\author[P. Mart\'in]{Pedro Mart\'in}
\author[D. Y\'a\~nez]{Diego Y\'a\~nez}
\address{Departamento de Matem\'aticas, Universidad de
Extremadura. 06071 Badajoz, Spain}
 \email{\dag \hspace{0.5pt} on March 7th 2014, pjimene@unex.es, dyanez@unex.es}
% \subjclass{46B20, 46C15, 52A10, 52A21}
\thanks{Partially supported by Junta de Extremadura grant GR15055 (partially financed with FEDER)}

\date{}

\begin{abstract}

Let $X$ be a real normed space with unit sphere S. We prove that $X$ is an inner product space if and only if there exists  a real number $\rho=\sqrt{(1+\cos\frac{2k\pi}{2m+1})/2}$, $(k=1,2,\ldots , m
;\:m=1,2,\ldots)$, such that every chord of $S$ that supports $\rho S$ touches $\rho S$ at its middle point. If this condition  holds, then every point $u\in
S$ is a vertex of a regular polygon that is inscribed in $S$ and circumscribed about $\rho S$.
\end{abstract}

\maketitle

\section{Introduction and notation}

Let $X$ be a real normed space with unit ball $B$ and unit sphere $S$.  $X$ is an inner product space (i.p.s.) if and only if every chord of $S$
supports a sphere homothetic to $S$ at its middle point, namely, if it  fulfils the ``nonbias''condition

$$u,v\in S \: \Rightarrow \: \inf_{t\in[0,1]}\|(1-t)u+tv\|=
\|\tfrac12u+\tfrac12v\|.$$

\noindent (\cite{GS}; see \cite{Am},  p. 29, where
this result is used to establish many characterizations of
i.p.s.).
%In geometrical terms this condition states that every chord of $S$
%supports an homothetic sphere of $S$ at its middle point.
But in order to characterize an i.p.s. we can only consider the chords of $S$  that supports $\rho S$ %$\tfrac12 S$
at its middle point  for some $\rho\in (0,1)$. Namely, given the following property (P-$\rho S$ from now on)
%\begin{equation*}
%u,v\in S,\: \inf_{t\in[0,1]}\|tu+(1-t)v\|=\tfrac12 \: \Rightarrow
%\: \tfrac12 u+\tfrac12 v\in \tfrac12 S. \tag{1}
%\end{equation*}

\begin{equation*} u,v\in S,\: \inf_{t\in[0,1]}\|(1-t)u+tv\|=\rho \:
\Rightarrow \: \tfrac12 u+\tfrac12 v\in \rho S, \tag{P-$\rho S$}
\end{equation*}
%\noindent or equivalently, every chord of $S$  that
%supports $\tfrac12 S$, touches $\tfrac12 S$ at its middle point.

\noindent
$X$ is an i.p.s.\hspace{-3pt} if and only if (P-$\rho S$) holds for $\rho=\tfrac12$ (see \cite{BY1}) or
for any real number $\rho$ such that (see \cite{BY2})
$$
0<\rho<1, \quad
\rho\neq\sqrt{(1+\cos\tfrac{2k\pi}{n})/2},\quad
(2k<n;\:n=3,4,...).
$$

%I.e.,  that $X$ is an i.p.s. if and only if, for any of these
%values of $\rho$,
%\begin{equation*} u,v\in S,\: \inf_{t\in[0,1]}\|tu+(1-t)v\|=\rho \:
%\Rightarrow \: \tfrac12 u+\tfrac12 v\in \rho S. \tag{2}
%\end{equation*}

%In geometrical terms, every chord of $S$ that supports that
%supports $\rho S$, touches $\rho S$ at its middle point.

The aim of  this paper is to prove (Theorem \ref{teo}) that
$X$ is an i.p.s. if and only if  (P-$\rho S$) holds for a real number on the set

$$
M=\left\{ \rho\in (0,1) / \, \rho=\sqrt{(1+\cos\tfrac{2k\pi}{2m+1})/2}:\; k=1,2,\ldots , m;\:m=1,2,\ldots \right\}.
$$

%I.e., we proved  that $X$ is an i.p.s. if and only if, for any of
%these values of $\rho$,
%\begin{equation*} u,v\in S,\: \inf_{t\in[0,1]}\|tu+(1-t)v\|=\rho \:
%\Rightarrow \: \tfrac12 u+\tfrac12 v\in \rho S. \tag{$\rho S$}
%\end{equation*}
%

It is known that $X$ is an i.p.s.\hspace{-6pt} if and only if so are its 2-dimensional subspaces.
This fact and the nature of the property (P-$\rho S$)  allow us
to consider that $X$ is a real 2-dimensional space from now on.

Given $u=(u_1,u_2)$ and $v=(v_1,v_2)$ in $X$, $[u,v]$ denotes the segment meeting $u$ and $v$, and
$u\prec v$ means that $u$ precedes $v$ in the positive orientation (counterclockwise) of
$X$, i.e., the following expression is positive:
$$
u\wedge v:=u_1v_2-u_2v_1.
$$

We say that $u$ is orthogonal to $v$ in the sense of Birkhoff (\cite{Bi}, \cite{Ja}), denoted by  $u\perp v$, if
$$
\|u\|\leq\|u+\lambda v\|\quad \forall \lambda\in\mathbb{R}.
$$
In other words, $u\perp v$ if and only if the homothetic copy of $S$ with scale factor $\|u\|$ is supported by the line $\{u+\lambda v:\lambda\in\mathbb{R}\}$ at $u$.
If  (P-$\rho S$) holds, and $[u,v]$ supports $\rho S$ for some $u,v\in S$, then  $u+v\perp v-u$.
%\end{rem}	

 If $u,v\in S$, $u\prec v$, %such that $u=s(\alpha)$ and $v=s(\beta)$ $(0\leq \alpha<\beta\leq 2\pi$),
$B_{u}^{v}$ denotes the sector of $B$ bounded by $u$, $v$, and the arc of $S$ from $u$ to $v$ (counterclockwise); and $T_{u}^{v}$ denotes the subset of $B_{u}^{v}$ bounded by $[u,v]$ and the arc.
If $\mathcal{A}(C)$ denotes the area of a set $C$, since  $u\wedge v$ is the area of the parallelogram defined by $u$ and $v$, then obviously $\mathcal{A}(B_u^v)=\frac12 (u\wedge v) + \mathcal{A}(T_u^v)$.

This paper is organized as follows. Some preliminary lemmas related to useful maps from $[0,2\pi]$ to $S$ are studied in Section \ref{prelemma}. The concepts of $\rho$-ellipses and $\rho$-polygons, already used in \cite{BY2}, are reintroduced in Section \ref{ellipsepolygon}, and some special results for $\rho\in M$ are proved in Section \ref{specialcase}. The main Theorem \ref{teo} is obtained in Section \ref{maintheo}, and an open problem is proposed in Section \ref{conclusion}.

\section{preliminary lemmas}\label{prelemma}

%%%%%%%%%%%%%%%%%%%NOTACION%%%%%%%%%%%%%%%%%%%%%

%%%%%%%%%%%%%%%%%%%%% u^* ;  REGULAR;   ORTOGONAL ;  LOS MU  %%%%%%%%%%%%%%%%%%%%%%%%%%
Along this section, some essential maps from $[0,2\pi]$ to $S$ and their properties are introduced. The first lemma and its proof appear in \cite{BY2}.

\begin{lem}\label{u} Let $0<\rho<1$. For any $u\in S$, there is a unique $u^*\in S$, $u\prec u^*$,
such that $[u,u^*]$ supports $\rho S$. The map $u\in
S\to u^*\in S$ is a homeomorphism, and $u\prec v$ implies that
$u^*\prec v^*$.
\end{lem}

%\begin{proof}
%The proof, that is very intuitive and not difficult, we can see
%in\cite{BY2}Lemma 2.1.
%\end{proof}

%\begin{rem}\label{2}
%$[u,u^*]$ soporta $\rho S$ significa que $[u,u^*]\cap\rho
%S\neq\emptyset$ and, para cada $x\in[u,u^*]\cap\rho S$, se tiene que $x\perp
%u^*-u$.
%\end{rem}

The next lemma summarizes a set of properties that are consequences of  (P-$\rho S$).

\begin{lem}\label{r} Let $0<\rho<1$. If $X$ fulfils \normalfont{(P-$\rho S$)}, then:
\begin{enumerate}
\item $X$ is regular (strictly convex and smooth).\label{r1}
\item For any $u\in S$, there is an unique $u^{\perp}\in S$, $u\prec u^{\perp}$, such that $u\perp u^{\perp}$ and the map $u\in S\to u^{\perp}\in S$ is a homeomorphism. If $v\in S$ and $u\prec v$, then $u^{\perp}\prec v^{\perp}$.\label{b}
\item \label{mu}For any $u\in S$, there exists an unique $\mu >0$,
such that  $\rho (u-\mu u^{\perp})$ and $\rho (u+\mu u^{\perp})$ belong to $S$.
\end{enumerate}
\end{lem}

\begin{proof}
Property (P-$\rho S$) implies that $[u,u^*]$ supports $\rho S$  at  $\tfrac12 u+\tfrac12 u^*$. Due to this fact and using Lemma \ref{u}, the proof of (1) and (2) presented in \cite{BY1} for the case $\rho=\tfrac12$ can be applied for every $0<\rho<1$.

Having in mind (2) and fixed $u\in S$, the convexity of the function $F:\lambda\in \mathbb{R}\to F(\lambda)=\|\rho u+\lambda \rho u^{\perp}\|$ implies that there exist only two real numbers $\mu,\nu\in \mathbb{R}_{+}$ such that $\|\rho u-\mu \rho u^{\perp}\|=\|\rho u+\nu \rho u^{\perp}\|=1$. Applying Lemma \ref{u}  to $\rho u-\mu \rho u^{\perp}$ and (P-$\rho S$), it is concluded that $\mu=\nu$.
\end{proof}

%\begin{rem}\label{3}
%When $X$ fulfils  $(\rho S)$, we can say that the segment $[u,u^*]$ is ``tanget" to $\rho S$  at  $\tfrac12 u+\tfrac12 u^*$.
%\end{rem}

%\begin{lem}\label{b} If $X$ fulfils ($\rho S$) then, for any $u\in S$, there is an unique $u^{\perp}\in S$, $u\prec u^{\perp}$,
%the map $u\in S\to u^{\perp}\in S$ is a homeomorphism, and $u\prec
%v$ implies that $u^{\perp}\prec v^{\perp}$.
%\end{lem}

%\begin{proof}
%The proof, which is not difficult, is given in detail for the case
%$\rho=\tfrac12$ in \cite{BY1} Lemma 2.6.
%\end{proof}

%\begin{lem}\label{mu}If $X$ fulfils ($\rho S$) then, for any $u\in S$, exist an unique real number $\mu >0$,
%such that, the points $\rho(u-\mu u^{\perp}),\rho(u+\mu
%u^{\perp})\in S$.
%\end{lem}
%
%\begin{proof}
%
%Dado $u\in  S$, de la convexidad de la  funcion:
%$\lambda\in \mathbb{R}\to\|\rho u+\lambda \rho u^{\perp}\|$,
%se sigue que existen \'unicos $\mu,\nu\in \mathbb{R}_{+}$,tales $\|\rho u-\mu \rho u^{\perp}\|=\|\rho u+\nu \rho u^{\perp}\|=1$.
%Como $X$ fulfils $(\rho S)$, y teniendo en cuenta el Lemma \ref{b} , se tiene que $\mu=\nu$.
%
%\end{proof}

%%%%%%%%%%%%%%%%%%% PARAMETRIZACION %%%%%%%%%%%%%%%%%%%%

Since $S$ is a convex curve, the following \emph{natural parametrization} $s$ is continuous and of
bounded variation
$$\begin{array}{rccl}
s:& [0,2\pi] &  \longrightarrow& S\\
&\theta &\mapsto & s(\theta)=(s_1(\theta),s_2(\theta)),
\end{array}$$
where $(s_1(\theta),s_2(\theta))=\|(\cos\theta,\sin\theta)\|^{-1}(\cos\theta,\sin\theta).$
%will be a ``natural map'' for $S$, i.e., a map such that
%$s(\theta)=(s_1(\theta),s_2(\theta))$ is the point of $S$ that
%makes an angle $\theta$ with a given point $(s_1(0),s_2(0))$ of
%$S$, measured with the positive orientation of the plane $X$. In
%other words, if $s(0)=\|(1,0)\|^{-1}(1,0)$, then
%$$s(\theta)=\|(\cos\theta,\sin\theta)\|^{-1}(\cos\theta,\sin\theta).$$

And as a consequence of Lemma \ref{u},  the following parametrization (non-natural, in
general) is also continuous and of
bounded variation
$$s^{*}:\theta\in[0,2\pi]\to s^{*}(\theta):=s(\theta)^*\in S.$$
%where $s^*(\theta)$ is the unique point of $S$ such that $s(\theta)\prec
%s^*(\theta)$ and the segment $[s(\theta),s^*(\theta)]$ supports
%$\rho S$.

Moreover, if $X$ fulfils (P-$\rho S$), the continuity and bounded variation hold for the parametrizations $s^\perp$, $\rho(s+\mu s^{\perp})$, and $\rho(s-\mu s^{\perp})$, and for the application $\mu s^{\perp}$ (by (\ref{b}) and (\ref{mu}) of Lemma \ref{r}) defined as follows:
%following maps are continuous and of
%bounded variation by (\ref{b}) and (\ref{mu}) in Lemma \ref{r}
$$
\begin{array}{rccl}
s^\perp:              & \theta\in[0,2\pi] &\to & s^\perp(\theta):=s(\theta)^{\perp}\in S,\\
\rho(s-\mu s^{\perp}):& \theta\in[0,2\pi] &\to & \rho(s(\theta)-\mu(\theta)
s^\perp(\theta))\in S,\\
\rho(s+\mu s^{\perp}):& \theta\in[0,2\pi] &\to &\rho( s(\theta)+\mu(\theta)
s^\perp(\theta))\in S, \\
\mu s^{\perp}:        & \theta\in[0,2\pi] &\to &\mu(\theta) s^{\perp}(\theta)\in X,
\end{array}$$
 %Para la primera en virtud del  Lemma \ref{b} y para las restantes en virtud del Lemma \ref{mu} .
where $\mu(\theta)$ is the real number considered for $u=s(\theta)$ in (\ref{mu}) of Lemma \ref{r}.
Therefore all the Riemann-Stieltjes integrals that we shall write
from now on make sense. For example, if $t$ is any of the parametrizations of $S$ introduced above, and $u=t(\alpha)$ and $v=t(\beta)$ $(0\leq \alpha<\beta\leq 2\pi$), then
\begin{equation*}\label{area}
\mathcal{A}(B_u^v)=\frac12\int_\alpha^\beta t(\theta)\wedge dt(\theta)=\frac12\int_{\alpha}^{\beta}[t_1(\theta)dt_2(\theta)-t_2(\theta)dt_1(\theta)].\tag{A}
\end{equation*}

%$$\mathcal{A}(B)=\frac12\int_0^{2\pi}
%s(\theta)\wedge ds(\theta)=
%\frac12\int_0^{2\pi}[s_1(\theta)ds_2(\theta)-s_2(\theta)ds_1(\theta)],$$
% thah gives the area of $B$.

\begin{lem}\label{i0}
Let $0<\rho<1$ and $X$ fulfil \normalfont{(P-$\rho S$)}. Let $s:\theta\in[0,2\pi]\to s(\theta)\in S$ be a natural parametrization for
$S$, and $s^{\perp}(\theta)$ and $\mu s^{\perp}(\theta)$ as they are defined above. Then, for any $0\leq\alpha<\beta\leq 2\pi$:
\begin{enumerate}
	\item $\int_\alpha^\beta \mu(\theta)s^{\perp}(\theta)\wedge ds(\theta)=0.$\label{i01}
	\item $\int_\alpha^\beta s(\theta)\wedge d s^{\perp}(\theta)=s(\beta)\wedge s^{\perp}(\beta)-s(\alpha)\wedge s^{\perp}(\alpha).$\label{i02}
	\item $\int_\alpha^\beta s(\theta)\wedge d[\mu(\theta) s^{\perp}(\theta)]=s(\beta)\wedge \mu(\beta)s^{\perp}(\beta)-s(\alpha)\wedge \mu(\alpha)s^{\perp}(\alpha).$\label{i03}
	\end{enumerate}
\end{lem}

\begin{proof}
Let $\alpha=\theta_0<\theta_1<...<\theta_n=\beta$ be a partition of $[\alpha,\beta]$. By the Mean Value Theorem, there exist $\eta_1, \eta_2,...,\eta_n\in \mathbb{R}$ such that
$\theta_0\leq\eta_1\leq\theta_1\leq...\leq\theta_{n-1}\leq\eta_n\leq\theta_n$ and
$s^{\perp}(\eta_k)\wedge[s(\theta_k)-s(\theta_{k-1})]=0$. Thus, the Riemmann-Stieltjes sum related to this partition is equal to $0$, and (\ref{i01}) holds.

(\ref{i02}) and (\ref{i03}) result from the integration by parts of $\int_\alpha^\beta d[s(\theta)\wedge s^{\perp}(\theta)]$ and of $\int_\alpha^\beta d[s(\theta)\wedge \mu s^{\perp}(\theta)]$, respectively, and (\ref{i01}).
\end{proof}

\begin{lem}\label{T}
Let $0<\rho<1$. If  $X$ fulfils \normalfont{(P-$\rho S$)}, the function $u\in S \to  \mathcal{A}(T_u^{u^*})$ is constant.
\end{lem}

\begin{proof}
 Let $u,v\in S$, $u\prec v$, and
$s:\theta\in [0, 2\pi]\to s(\theta)\in S$ be a parametrization of $S$. By (3) of Lemma \ref{r}, there exist $0\leq \alpha<\beta\leq 2\pi$ such that
$$
u=\rho[s(\alpha)-\mu(\alpha)s^\perp(\alpha)],
\quad u^*=\rho[s(\alpha)+\mu(\alpha)s^\perp(\alpha)]
$$
$$
v=\rho[s(\beta)-\mu(\beta)s^{\perp}(\beta)],
\quad v^*=\rho[s(\beta)+\mu(\beta)s^\perp(\beta)].
$$
By (\ref{area}),
$$\mathcal{A}(B_u^v)=\frac{\rho ^2}2\int_{\alpha}^{\beta} [s(\theta)-\mu(\theta)s^{\perp}(\theta)]\wedge
d [s(\theta)-\mu(\theta)s^{\perp}(\theta)],$$
$$\mathcal{A}(B_{u^*}^{v^*})=\frac{\rho ^2}2\int_{\alpha}^{\beta} [s(\theta)+\mu(\theta)s^{\perp}(\theta)]\wedge
d [s(\theta)+\mu(\theta)s^{\perp}(\theta))].$$

 \noindent Therefore, (\ref{i01}) and (\ref{i03}) of Lemma \ref{i0} imply that
 \begin{multline*}\mathcal{A}(B_{u^*}^{v^*})-\mathcal{A}(B_u^v)=\\
 =\rho ^2 \int_{\alpha}^{\beta} s(\theta)\wedge d[\mu(\theta)s^{\perp}(\theta)]+ \rho ^2\int_{\alpha}^{\beta}\mu(\theta)s^{\perp}(\theta)\wedge d s(\theta)=\\
%\rho ^2 \int_{\alpha}^{\beta} s(\theta)\wedge d[\mu(\theta)s^{\perp}(\theta)]=%\rho ^2\int_{\alpha}^{\beta} d[s(\theta)\wedge\mu(\theta)s^{\perp}(\theta)]=\\
=\rho ^2[s(\beta)\wedge \mu(\beta)s^{\perp}(\beta)-s(\alpha)\wedge \mu(\alpha)s^{\perp}(\alpha)]= \\
=\frac12(v\wedge v^*)-\frac12(u\wedge u^*).
\end{multline*}

If $v\prec u^*$, it is easy to check that  $\mathcal{A}(B_u^{u^*})=\mathcal{A}(B_u^v)+ \mathcal{A}(B_{v}^{u^*})$ and $\mathcal{A}(B_v^{v^*})=\mathcal{A}(B_v^{u^*})+ \mathcal{A}(B_{u^*}^{v^*})$. Similarly, $\mathcal{A}(B_u^v)=\mathcal{A}(B_u^{u^*})+ \mathcal{A}(B_{u^*}^v)$ and $\mathcal{A}(B_{u^*}^{v^*})=\mathcal{A}(B_{u^*}^v)+ \mathcal{A}(B_v^{v^*})$  if $u^* \prec v$.
In both cases, it is verified that
$$
\mathcal{A}(B_{u^*}^{v^*})-\mathcal{A}(B_u^v)=\frac12(v\wedge v^*)-\frac12(u\wedge u^*)+\mathcal{A}(T_v^{v^*})-\mathcal{A}(T_u^{u^*}),
$$
and it is concluded that $\mathcal{A}(T_v^{v^*})-\mathcal{A}(T_u^{u^*})=0$.

%\vspace{2cm} mientras que si $u^* \prec v$, se verifica
%que: $$A(B_{u v})=A(B_{u u^*})+ A(B_{u^* v}) \quad\mbox{y}\quad
%A(B_{u^* v^*})=A(B_{u^*v})+ A(B_{v v^*}),$$ para concluir que
%$A(T_{uu^*})=A(T_{vv^*})$.

\end{proof}

%%%%%%%%%%%%%%% P-POLIGONOS Y P- ELIPSE %%%%%%%%%%%%%%

\section{$\rho$-ellipses and $\rho$-polygons} \label{ellipsepolygon}
 Fixed $0<\rho<1$ and $u\in
S$, the {\bf $\rho-$polygon} associated to $u$ is the  set of ordered points $P_u=\{u_1, u_2, u_3,...\}$ of $S$ and the segments $[u_i,u_{i+1}]$, such that $u=u_1\prec u_2\prec u_3,...$, and every  segment $[u_i,u_{i+1}]$ supports $\rho S$, i.e., $u_{i+1}=u_i^*$. Each $u_i\in P_u$ is a \emph{vertex} of $P_u$ and each $[u_i,u_{i+1}]$ is a \emph{side} of $P_u$. The {\bf $\rho$-ellipse} $C_u$ associated to $u\in S$ is the unique ellipse centered at $(0,0)$ that contains the points $u$, $\frac1{2\rho}(u+u^*)$, and $u^*$ (see Figure \ref{dibujoortgonal}).

 Some properties and examples of $\rho-$polygons and $\rho$-ellipses are presented in \cite{BY2}. For instance, if $S$ is an ellipse and $\rho=\sqrt{(1+\cos\frac{2k\pi}{n})/2}$ with $\frac{k}{n}$ irreducible, it is proved that $P_u$ is convex with $n$ vertices for $k=1$ and $n=3,4,...$; $P_u$ is star-shaped with $n$ vertices for either $k=2,...,\frac{n}2-1$ when $n$ is even, or for $k=2,...,\frac{n-1}2$ when $n$ is odd; $P_u$ is dense if $\rho$ is not in any of the previous cases. On the other hand, if $S$ is the unit sphere of $\ell_\infty^2$ (with vertices $\{(\pm1,\pm1)\}$) and $\rho=\frac12$, then $P_{(1,0)}=\{(0,\pm 1),(\pm1,0)\}$; and if $u\notin \{(0,\pm1), (\pm1,0)\}$, then $P_u$ has infinite vertices, but it is not dense in $S$ (the points $\{(0,\pm1), (\pm1,0)\}$ are the unique points of accumulation of $P_u$). And the same happens for $l_p^2$ ($p>2$) and $\rho=(\tfrac12)^{\frac{p-1}{p}}$.

%\begin{ex}\label{e1} Elementary calculations show that if $S$ is a circumference,
%i.e. $X=\ell_2^2$, and $u$ is any point of $S$, then:

%\begin{enumerate}
%%\item If $\rho=\sqrt{(1+\cos\frac{2\pi}{n})/2}$, ($n=3,4,...$), the $\rho$-polygon
%$P_u$ is convex  with $n$ vertices.
%
%\item If $\rho=\sqrt{(1+\cos\frac{2k\pi}{n})/2}$, with $\frac{k}{n}$
%irreducible and $k=2,...,\frac{n}2-1$ when $n=4,6,...$, and
%$k=2,...,\frac{n-1}2$ when $n=3,5,...$, the $\rho$-polygon $P_u$
%is starshaped  with $n$ vertices.
%
%\item If $\rho$ is not in the previous cases, the $\rho$-polygon $P_u$
%is dense in $S$.
%\end{enumerate}
%\end{ex}

%
%\begin{rem}\label{4} The above remains true when $S$ is an ellipse, i.e.,
%when $X$ is an i.p.s..
%\end{rem}

%\begin{ex}\label{e2} If $S$ is the square of vertices
%$(\pm1,\pm1)$, i.e. $X=\ell_\infty^2$, and $\rho=\frac12$, then:
%
%\begin{enumerate}
%\item $P_{(1,0)}=\{(1,0),(0,1),(-1,0),(0,-1)\}$.
%
%\item If $u\in S$ is not a vertex of the above $\rho$-polygon, $P_u$
%has infinite vertices but it is not dense in $S$ (the four points
%of the above $\rho$-polygon are the unique points of accumulation
%of $P_u$).
%\end{enumerate}
%\end{ex}

%%%%%%%%%%%%%%% PROPIEDADES P-POLYGON  %%%%%%%%%%%%%%%%%%%%%%%%%%%%%%

The next lemma presents some properties of $\rho$-polygons.

\begin{lem}\label{pol1} Let $0<\rho<1$. Then:
\begin{enumerate}
\item If there exists  $u\in S$ such that $P_u$ is dense in $S$, then $P_v$ is dense in $S$ for any $v\in S$.

\item If there exist  $u,v\in S$ such that $P_u$ and $P_v$ have a finite number of points, then both  have the same number of vertices, and either both are convex or both are star-shaped. Besides, in this last case the number of vertices that are (geometrically situated on $S$) between $u_i$ and $u_{i+1}$ is equal to the number of vertices between $v_i$ and $v_{i+1}$.

\item If $P_u=\{u_1, u_2, u_3,...\}$, then $-P_u=\{-u_1, -u_2,-u_3,...\}=P_{-u}$.

\item If $P_u$ has an odd number of vertices, %for $u\in S$,
then $P_u \cap P_{-u}=\emptyset$.
\end{enumerate}
\end{lem}

\begin{proof}
The statement (1) is proved in Lemma $2.2$ of \cite{BY2}.

In order to prove (2), let us consider $P_u=\{u_1, u_2,...,u_n \}$ and $P_v=\{v_1,v_2,...,v_m\}$ with $m>n$. Let us assume without loss of generality that $u_1\prec v_1\prec u_p$, and there is not any other vertex of $P_u$ between $u_1$ and $u_p$ ($p<n$). By Lemma \ref{u},
$$u_2\prec v_2\prec u_{p+1},\quad \ldots ,\quad u_n\prec v_n\prec u_{p+n-1},\quad u_1\prec v_{n+1}\prec u_p,\quad\ldots ,$$ and it is concluded that $m=kn$, $k\in \mathbb{N}$.
Thus, $u_1\prec\ v_1 \prec u_p$ and $u_1\prec v_{n+1}\prec u_p$.
But if $v_1\prec v_{n+1}$, then  $v_1\prec v_{kn+1}$ ($\forall k\in \mathbb{N}$) by Lemma \ref{u}, and this is a contradiction because $v_1=v_{m+1}=v_{kn+1}$. The analysis is similar if $v_{n+1}\prec v_1$.

Let us assume now that  $P_u$ is convex, $P_v$ is star-shaped, and each set has $n$ vertices. Then,
$u_1\prec v_1\prec u_2\prec v_2 \prec u_3\prec ...\prec u_n\prec v_n \prec u_1,$ leads to a contradiction.

For the last assertion of (2), it is enough to consider that  $u_i\prec u_m \prec u_{i+1}$ implies (Lemma \ref{u})  $u_{i+1}\prec u_{m+1} \prec u_{i+2}$.

The equality $-P_u=P_{-u}$ of (3) is a consequence of the symmetry of $S$.

Let us see (4). Let us assume that $P_u$ has an odd number of vertices and that there exist $-u_p\in -P_u$ and $u_q\in P_u$ such that $-u_p=u_q$. Then, $(-u_p)^*=-u_{p+1}$ (by symmetry of $S$) and $(-u_p)^*=u_{q+1}$ (by the construction of $P_u$). I.e.,
$-u_{p+1}=u_{q+1}$, and, in general, $-u_{p+k}=u_{q+k}$, $(k\in\mathbb{N}).$ Consequently, $P_u$ is symmetric, and this is not possible because it has an odd number of vertices.
\end{proof}

%%%%%%%%%%%%%%%%%%%% SUMAS DE TRIANGULOS CONSTANTE %%%%%%%%%%%%%%%%%%%%%%%%%%%%%%%

\begin{lem}\label{ig}
Let $u,v\in S$ such that $P_u=\{u_1, u_2,...,u_n \}$ and $P_v=\{v_1,v_2,...,v_n\}$ for some $0<\rho<1$.  If $X$ fulfils \normalfont{(P-$\rho S$)}, then
$$u_1\wedge u_2+...+u_{n-1}\wedge u_n+u_n\wedge u_1\, =\, v_1\wedge v_2+...+v_{n-1}\wedge v_n+v_n\wedge v_1.$$
\end{lem}

\begin{proof}
Let  $r$ be the constant number of vertices of $P_u$ that are (geometrically situated) between $u_i$ and $u_{i+1}$ (or, equivalently, between $v_i$ and $v_{i+1}$ by (2) of Lemma \ref{pol1}).
The statement is a consequence of Lemma \ref{T} and these equalities
\begin{multline*}
(r+1)\mathcal{A}(B)=\mathcal{A}(B_{u_1}^{u_2})+...+\mathcal{A}(B_{u_n}^{u_1})=\\
\frac12 (u_1\wedge u_2)+\mathcal{A}(T_{u_1}^{u_2})+...+\frac12 (u_n\wedge u_1)+\mathcal{A}(T_{u_n}^{u_1}),
\end{multline*}
\begin{multline*}
(r+1)\mathcal{A}(B)
=\mathcal{A}(B_{v_1}^{v_2})+...+\mathcal{A}(B_{v_n}^{v_1})=\\
\frac12 (v_1\wedge v_2)+\mathcal{A}(T_{v_1}^{v_2})+...+\frac12 (v_n\wedge v_1)+\mathcal{A}(T_{v_n}^{v_1}).
\end{multline*}
\end{proof}

%%%%%%%%%%%%%%%%%%%% PROPIEDADES P-ELLIPSE %%%%%%%%%%%%%%%%%%%%%%

%\begin{rem}\label{5}

%\end{rem}

The following result %(Lemmas $3.2$ and  $3.3$ in \cite{BY2})
presents some properties about $\rho$-ellipses and spheres that are \textit{tangent}. It is said that $C_u$  and $S$ are \textit{tangent} at $v\in S\cap C_u$ if both curves have the same supporting line at $v$. If $X$ fulfils (P-$\rho S$), then $C_u$ and $S$ are tangent at $u$ if and only if the common supporting line at $u$ is
$$
\{u+\lambda [(1-2\rho^2)u+u^*]:\; \lambda \in \mathbb{R}\},
$$ that is, if and only if $u$ has the following property (see Figure \ref{dibujoortgonal})
\begin{equation}\label{support}
u\;\perp\;(1-2\rho^2)u+u^*. \tag{$*$}
\end{equation}
Likewise, $C_u$ and $S$ are tangent at $u^*\in S\cap C_u$ if and only if $u^*$ verifies
\begin{equation}\label{support2}
u^*\;\perp \; -u - (1-2\rho^2)u^*. \tag{$**$}
\end{equation}

\begin{figure}[ht]
	\begin{center}
		%\scalebox{0.6}{\includegraphics[viewport=156 366 356 579,clip]{dibujoortgonal.pdf}}\\
        %\scalebox{0.8}{\includegraphics[viewport=70 525 441 700,clip]{dibujoortgonalintentosconnorma4.pdf}}\\
        %\scalebox{0.8}{\includegraphics[viewport=100 515 530 700,clip]{dibujoortgonal2.pdf}}\\
        \scalebox{0.8}{\includegraphics[viewport=100 515 530 700,clip]{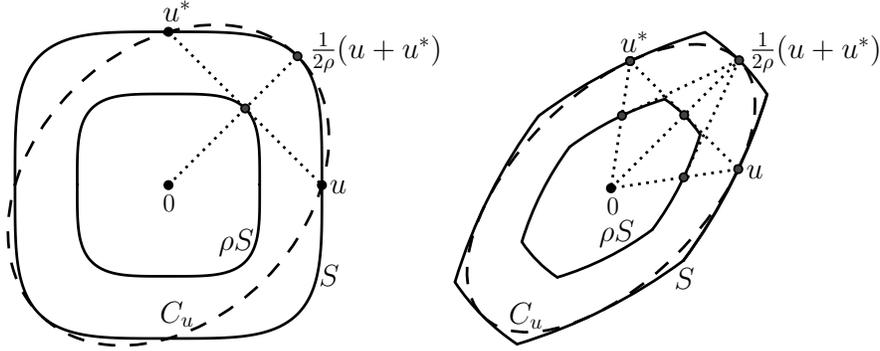}}\\
		\caption{$\rho$-ellipses $C_u$ associated to $u\in S$. On the right, $S$ and $C_u$ are tangent at $u$, $\frac{u+u^*}{2\rho}$, and $u^*.$}\label{dibujoortgonal}
	\end{center}
\end{figure}

\begin{lem}\label{elip}Let $0<\rho<1$. If $X$ fulfils  \normalfont{(P-$\rho S$)}, then:
\begin{enumerate}
 \item $C_u$ and $S$ are tangent at $\frac1{2\rho} (u+u^*)$ for every $u\in S$.

%\item If $C_u$  supports $S$ at $u$, then it also supports $S$ at $u^*$.

\item If $C_u$ and $S$ are tangent at $u\in S$ (equivalently, if $u$ verifies (\ref{support})), then $C_{u}$ and $S$ are tangent at every point of $P_u\cup P_w$, where $w=\frac1{2\rho} (u+u^*)$.

\item There exists $v\in S$ such that $C_v$ and $S$ are tangent at $v$.% and, consequently, it supports $S$  at all the points of $P_v\cup P_w$, where $w=\frac1{2\rho} (v+v^*)$).
\end{enumerate}
\end{lem}
\begin{proof}
Lemma $3.2$ in \cite{BY2} proves (1), (3), and that if $C_u$ and $S$ are tangent at $u$, then $C_u$ and $S$ are tangent at $u^*$. The proof of Lemma $3.3$ in \cite{BY2} can be applied for every $u\in S$ such that $C_u$ and $S$ are tangent at $u$, and (2) holds.
\end{proof}

%\begin{rem}\label{5}
%Observe that $C_u$  supports $S$ at $u\in S\cup C_u$ if and only if the line $\{u+\lambda (1-2\rho^2)u+u^*: \lambda \in \mathbb{R}\}$ supports both $C_u$ and $S$ at $u$, that is $u\perp  (1-2\rho^2)u+u^*$. Similarly, $C_u$ supports $S$ at $u^*\in S\cup C_u$ if and only if $u^*\perp u+  (1-2\rho^2)u^*$.
%\end{rem}

%\begin{proof}
%Is the Lemmas $3.2$ and  $3.3$ of \cite{BY2} .
%\end{proof}

%\begin{rem}\label{5} Puede verse en \cite{BY2} que el hecho de que la  $\rho$-ellipse $C_u$  sea tangente a $S$ en $u\in S$ significa que $u\in S\cup C_u$ y la recta $\{u+\lambda (1-2\rho^2)u+u^*: \lambda \in \mathbb{R}\}$ soporta (es tangente) a ambas curvas en dicho punto $u$, es decir: $u\perp  (1-2\rho^2)u+u^*$. Queremos ver que si  $X$ fulfils ($\rho S$), entonces esta  misma condicion se da para cada punto $u\in S$. En otras palabras, que si la cuerda de extremos $u, u^* \in S$ soporta a $\rho S$ en su punto medio, $\tfrac12(u+u^*)$, entonces $\rho u\in\rho S$ es el punto medio de una cuerda que soporta a $\rho S$ con extremo en $\tfrac1{2\rho}(u+u^*)\in S$. An\~{A}!`logamente, que lo sea en $u^*$ significa que $\rho u^*\in\rho S$ es el punto medio de una cuerda que soporta a $\rho S$ con extremo en $\tfrac1{2\rho}(u+u^*)\in S$, es decir:  $u^*\perp u+  (1-2\rho^2)u^*$.
%\end{rem}

%%%%%%%%%%%%%%%%%%%% UN PUNTO %%%%%%%%%%%%%%%%%%%%5

\section{$\rho$-ellipses and $\rho$-polygons: the special case $\rho\in M$}\label{specialcase}
We remind the reader the definition of $M$
$$M=\left\{ \rho\in (0,1) / \, \rho=\sqrt{(1+\cos\tfrac{2k\pi}{2m+1})/2}:\; k=1,2,\ldots , m;\:m=1,2,\ldots \right\}.$$

\begin{lem}\label{punto} Let $\rho \in M$. If $X$ fulfils \normalfont(P-$\rho S$), the following properties hold for every $v\in S$ such that $C_v$ and $S$ are tangent at $v$ (equivalently, for every $v$ that verifies (\ref{support})):
\begin{enumerate}

\item $P_v=\{v_1,v_2,...,v_n\}$ has $n=2m+1$ vertices, and $C_v$ and $S$ are tangent at every $v_i\in P_v$.%all of them belong to $C_v$.
%vertices que estan todos $\rho-$ellipse $C_v$.

\item If $w=\frac1{2\rho} (v+v^*)$, $P_w$ has $n=2m+1$ vertices. Such as vertices are the points $w_{i}=\tfrac1{2\rho}(v_i+v_{i+1})$,  and $C_v$ and $S$ are tangent at every $w_i\in P_w$.

\item If $k$ is an odd number, $P_{w}=P_{-v}$. If $k$ is an even number, $P_w=P_v$.

\item $v_1\wedge v_2=...=v_{n-1}\wedge v_n=v_n\wedge v_1$.

\item %All the sectors $B_{v_i}^{v_{i+1}}$ have the same area:\\
 $\mathcal{A}(B_{v_1}^{v_2})=\ldots=\mathcal{A}(B_{v_{n-1}}^{v_n})=\mathcal{A}(B_{v_n}^{v_1}).$

\item The vertices of $P_v$ and $P_{-v}$ split $B$ into $2n$
disjoint sectors of equal area.

\end{enumerate}

\end{lem}

\begin{proof}
Lema \ref{elip} ensures the existence of $v\in S$ such that $C_v$ and $S$ are tangent at every vertex of $P_v\cup P_w$.
Since $C_v$ is an ellipse, $P_v$ has $n=2m+1$ vertices for $\rho\in M$ (see Example 1 in \cite{BY2} or the comments at the beginning of Section \ref{ellipsepolygon}) and (1) holds, as well as (2).

It is easy to see that (3) and (4) are true when $S$ is an ellipse (see Figure \ref{dibujopoligonos7lados}). But in the general case, the vertices of $P_v$ and $P_w$ are always the vertices of $\rho$-polygons  inscribed in the $\rho$-ellipse $C_v$ and circumscribed about its homothetic ellipse of ratio $\rho$ (Lemma \ref{elip}). Therefore, (3) and (4) hold for every $S$.

% By construction of $P_v$ and $P_w$, the vertices of these
%$\rho$-polygons they are vertices of $\rho$-polygons unscribed at an
%ellipse and circunscrived to its homothetic of radius $\rho$ (for
%the fisex value of $\rho$), whit $w_i=\frac1{2\rho}
%(v_i+v_{i+1})$, and this property is true for $\rho$-ellipses.

%Statement (4) is true for ellipses, hence also in general, because the vertices of $P_v$ and
%$P_w$ are on the $\rho$-ellipse $C_v=C_w$

From (4) and Lemma \ref{T},  (5) is obtained.

%******************************
%
%
%Since $n=2m+1$ is an odd number, then (see (4) in Lemma \ref{pol1}),
%the vertices of $P_v\cup P_{-v}$ determine $2n$ different vectors in $S\cap C_v$: $\{v_1$, ..., $v_n$, $v_{-1}$,..., $v_{-n}\}$. Applying (5) to $P_v$ and $P_{-v}$, it is concluded (6).
%
%*******************************

Let us see (6). Since $n=2m+1$ is an odd number, then (see (4) of Lemma \ref{pol1}),
the vertices of $P_v\cup P_{-v}$ determine $2n$ different vectors in $S\cap C_v$: $\{v_1$, ..., $v_n$, $-v_{1}$,..., $-v_{n}\}$.
For every $v_i\in P_v$, let $-v_{\sigma (i)}\in P_{-v}$ be such that  $v_i\prec -v_{\sigma(i)}$ and there is not any other vertex of  $P_v\cup P_{-v}$ between (counterclockwise) $v_i$ and $-v_{\sigma (i)}$ (see Figure \ref{dibujopoligonos7lados}).

Due to the symmetry of $S$ and (5), $\mathcal{A}(B_{v_i}^{v_{i+1}})=\mathcal{A}(B_{-v_{\sigma(i)}}^{-v_{\sigma(i+1)}})$. Using arguments similar to those in Lemma \ref{T}, it holds that:
$$
\begin{array}{rcl}
\mathcal{A}(B_{v_i}^{v_{i+1}})&=&\mathcal{A}(B_{v_i}^{-v_{\sigma(i)}})+\mathcal{A}(B_{-v_{\sigma(i)}}^{v_{i+1}}),\\
\mathcal{A}(B_{-v_{\sigma(i)}}^{-v_{\sigma(i+1)}})&=&\mathcal{A}(B_{-v_{\sigma(i)}}^{v_{i+1}})+\mathcal{A}(B_{v_{i+1}}^{-v_{\sigma(i+1)}}).
\end{array}$$
Hence, the $n$ disjoint sectors $B_{v_i}^{-v_{\sigma (i)}}$ have the same area. Since $n$ is an odd number (and again the symmetry of $S$), (6) holds.

%Procediendo de modo analogo para los vertores de $P_{-v}$, para cada $-v_i\in P_{-v}$ sea $v_{\gamma (i)}\in P_{v}$,  $ -v_i\prec v_{\gamma(i)}$,  y entre ellos no hay ningun otro vertice de  $P_v\cup P_{-v}$. Usando el mismo razonamiento anterior, se llega a la conclusion de que los  $n$ sectores disjuntos $B_{-v_iv_{\gamma (i)}}$ tienen el mismo area.

%Finalmente, basta tener en cuenta que $n$ es impar y la simetria de $S$ para concluir que las areas de unos y otros sectores disjuntos es la misma.

%************************************

\end{proof}

\begin{figure}[ht]
\begin{center}
%\scalebox{0.74}{\includegraphics[viewport=64 391 551 526,clip]{dibujopoligonos7lados.pdf}}\\
\scalebox{0.84}{\includegraphics[viewport=127 502 471 663,clip]{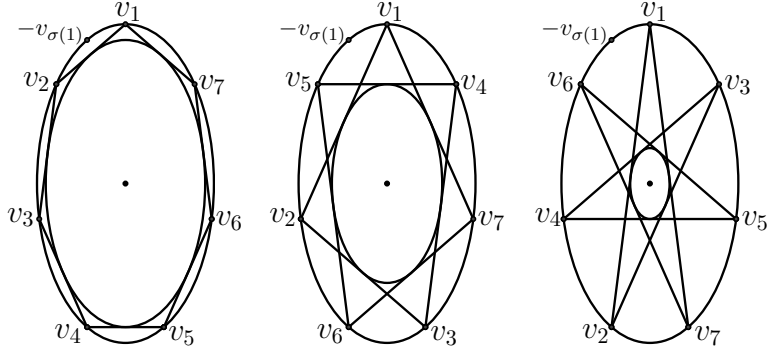}}\\
%%\hspace{1.5cm}
%%{\scalebox{0.7}{\includegraphics*[0,0][240,155]{Construccion_curvas_k1_a2_b2_sombra.eps}}}
\caption{$\rho$-polygon $P_v$ with $n=7$ and $k=1$ (left), $k=2$ (center), and $k=3$ (right).}\label{dibujopoligonos7lados}
\end{center}
\end{figure}

%%%%%%%%%%%%%%%%%%%%%%% TODOS LOS PUNTOS %%%%%%%%%%%%%%%%%%%%%

\begin{lem}\label{todos} Let $\rho \in M$. If $X$ fulfils \normalfont(P-$\rho S$), then  for every $u\in S$ it is verified that $P_u=\{u_1, u_2, u_3,...\}$ has $n=2m+1$ vertices and $C_u$ supports $S$ at every vertex of $P_u$. Besides,  the function  $u\in S\to u\wedge u^*$ is constant, and $\mathcal{A}(B_u^v)=\mathcal{A}(B_{u^*}^{v^*})$ for any $v\in S$ such that $u\prec v$.
\end{lem}

\begin{proof}We remind that if $X$ fulfils (P-$\rho S$) and $u\in S$, then $C_u$ and $S$ are tangent at $u$ if and only %if the common supporting line at $u$ is $$\{u+\lambda (1-2\rho^2)u+u^*: \lambda \in \mathbb{R}\},$$ that is, if and only if
$u$ has  property (\ref{support}). The proof is organized in four steps.
%$(\ref{support})$.
%\begin{equation}\label{support}
%u\perp(1-2\rho^2)u+u^*. \tag{*}
%\end{equation}

\textit{Step 1: there exists $v\in S$ such that the condition} (\ref{support}) \textit{is verified for every $z\in P_v \cup P_{-v}.$}

Since $S$ and the $\rho$-ellipses are symmetric, it is deduced (by Lemma \ref{elip}, Lemma \ref{punto}, and (3) and (4) of Lemma \ref{pol1}) that there exists $v\in S$ such that the condition (\ref{support}) is verified for every $z\in P_v \cup P_{-v}.$

%$C_v$ supports $S$ at every vertex $z\in P_v \cup P_{-v}$. In others words (Remark \ref{5}),
%	$$
%	z\perp(1-2\rho^2)z+z^*\quad \forall z\in P_v\cup P_{-v}.
%	$$

\medskip

Let $P_v=\{v_1,v_2,\dots,v_n\}$ be the polygon generated by $v$. As in Lemma \ref{punto}, %for any $v_i\in P_v$,
let us denote $-v_{\sigma (i)}$ to the unique point in $P_{-v}$ such that $v_i\prec -v_{\sigma(i)}$ and there is not any other vertex of  $P_v\cup P_{-v}$ between (counterclockwise) $v_i$ and $-v_{\sigma(i)}$.

\textit{Step 2: for every $i\in \{1,2\dots,n\}$, there exists $\bar{v}\in S$   such that   $\bar{v}$ verifies} (\ref{support}) and $v_i\prec\bar{v}\prec -v_{\sigma(i)}$.% $\bar{v}\perp(1-2\rho^2)\bar{v}+\bar{v}^*.$}

Without loss of generality, let us assume that $i=1$. Let us consider a parametrization
$s:\theta\in [0, 2\pi]\to s(\theta)\in S$ of $S$. By (3) of Lemma \ref{r}, there exist $0\leq \theta_1<\theta'_1\leq 2\pi$ such that
$$
v_1=\rho[s(\theta_1)-\mu(\theta_1)s^\perp(\theta_1)],  \quad -v_{\sigma(1)}= \rho[s(\theta'_1)-\mu(\theta'_1)s^\perp(\theta'_1)],
$$
 $$
 v_2=\rho[s(\theta_{1})+\mu(\theta_{1})s^\perp(\theta_{1})], \quad  -v_{\sigma(2)}= \rho[s(\theta'_{1})+\mu(\theta'_{1})s^\perp(\theta'_{1})].
 $$
It is proved just some lines below that
\begin{equation}\label{integral}
\int_{\theta_1}^{\theta'_1}[(1-\rho^2)s(\theta)+\rho^2\mu(\theta)s^\perp(\theta)]\wedge
d[s(\theta)-\mu(\theta)s^\perp(\theta)]=0,
\end{equation}
and as a consequence, there exists $\theta_1<\bar{\theta}<\theta'_1$ such that
$$
s(\bar{\theta})-\mu(\bar{\theta})s^\perp(\bar{\theta}) \; \perp \; (1-\rho^2)s(\bar{\theta})+\rho^2\mu(\bar{\theta})s^ \perp(\bar{\theta}).
$$
Thus the points
$$\bar{v}:=\rho[s(\bar{\theta})-\mu(\bar{\theta})s^\perp(\bar{\theta})], \quad \bar{v}^*:=\rho[s(\bar{\theta})+\mu(\bar{\theta})s^\perp(\bar{\theta})],$$
 verify $\bar{v}\;\perp\;(1-2\rho^2)\bar{v}+\bar{v}^*$, with $v_1\prec\bar{v}\prec -v_{\sigma(1)},$ as Step 2 claims.
\smallskip

In order to see (\ref{integral}), the integral is separated into four parts as follows:
\begin{multline*}\int_{\theta_1}^{\theta'_1}[(1-\rho^2)s(\theta)+\rho^2\mu(\theta)s^\perp(\theta)]\wedge
d[s(\theta)-\mu(\theta)s^\perp(\theta)]=\\
(1-\rho^2)\int_{\theta_1}^{\theta'_1}s(\theta)\wedge d s(\theta)
-(1-\rho^2)\int_{\theta_1}^{\theta'_1}s(\theta)\wedge
d[\mu(\theta)s^\perp(\theta)]\\
+\rho^2\int_{\theta_1}^{\theta'_1}\mu(\theta)s^\perp(\theta)\wedge
d s(\theta)
-\rho^2\int_{\theta_1}^{\theta'_1}\mu(\theta)s^\perp(\theta)\wedge
d[\mu(\theta)s^\perp(\theta)].
\end{multline*}
Let us denote $w_1$ and $-w_{\sigma(1)}$, respectively, to $\frac{1}{2\rho}(v_1+v_2)$ and $-\frac{1}{2\rho}(v_{\sigma(1)}+v_{\sigma(2)})$. Then
$w_{1}=s(\theta_1)$ and $-w_{\sigma(1)}=s(\theta'_1)$. By the calculus of area (see (\ref{area})), the first part is
%%%%%%%%%%%%%%%PRIMERA INTEGRAL%%%%%%%%%%%%%%%%%%%%
$$(1-\rho^2)\int_{\theta_1}^{\theta'_1}s(\theta)\wedge
ds(\theta)=2(1-\rho^2) \; \mathcal{A}(B_{w_1}^{-w_{\sigma(1)}}).$$
%%%%%%%%%%%%%%% SEGUNDA INTEGRAL%%%%%%%%%%%%%%%%%%%%
By (3) of Lemma \ref{i0} and (4) of Lemma \ref{punto}, the second part is
\begin{multline*}\int_{\theta_1}^{\theta'_1}s(\theta)\wedge
d[\mu(\theta)s^\perp(\theta)]=\\
s(\theta'_1)\wedge\mu(\theta'_1)s^\perp(\theta'_1)-s(\theta_1)\wedge\mu(\theta_1)s^\perp(\theta_1)=\\
\tfrac1{2\rho^2}[ (-v_{\sigma(1)}\wedge -v_{\sigma(2)}) - (v_1\wedge v_2)]=0.
\end{multline*}
%%%%%%%%%%%%%%TERCERA INTEGRAL %%%%%%%%%%%%%%%%%%
And by (1) of Lemma \ref{i0}, the third part is
$$
\rho^2 \int_{\theta_1}^{\theta'_1}\mu(\theta)s^\perp(\theta)\wedge d s(\theta)=0.
$$
%%%%%%%%%%%%%%%CUARTA INTEGRAL%%%%%%%%%%%%%%%%%%%%

Regarding the last part of the decomposition, using again the calculus of the area (\ref{area}), and the same statements of Lemma \ref{i0} and Lemma \ref{punto}, it is deduced that
\begin{multline*}\mathcal{A}(B_{v_1}^{-v_{\sigma(1)}})=\\
\frac{\rho^2}2\int_{\theta_1}^{\theta'_1}
[s(\theta)-\mu(\theta)s^\perp(\theta)]\wedge
d[s(\theta)-\mu(\theta)s^\perp(\theta)]=\\
\frac{\rho^2}2\int_{\theta_1}^{\theta'_1}s(\theta)\wedge
d s(\theta)
-\frac{\rho^2}2\int_{\theta_1}^{\theta'_1}s(\theta)\wedge d[\mu(\theta)s^\perp(\theta)]\\
-\frac{\rho^2}2\int_{\theta_1}^{\theta'_1}\mu(\theta)s^\perp(\theta)\wedge
d s(\theta)
+\frac{\rho^2}2\int_{\theta_1}^{\theta'_1}\mu(\theta)s^\perp(\theta)\wedge d[\mu(\theta)s^\perp(\theta)]=\\
\rho^2 \mathcal{A}(B_{w_1}^{-w_{\sigma(1)}})
+\frac{\rho^2}2\int_{\theta_1}^{\theta'_1}
\mu(\theta)s^\perp(\theta)\wedge d[\mu(\theta)s^\perp(\theta)],
\end{multline*}
and concluded that
$$\rho^2 \int_{\theta_i}^{\theta'_i}
\mu(\theta)s^\perp(\theta)\wedge d[\mu(\theta)s^\perp(\theta)]= 2 \mathcal{A}(B_{v_1}^{-v_{\sigma(1)}})-2\rho^2 \mathcal{A}(B_{w_1}^{-w_{\sigma(1)}}).
$$
Since  $\mathcal{A}(B_{v_1}^{-v_{\sigma(i)}})=\mathcal{A}(B_{w_1}^{-w_{\sigma(1)}})$  by (3) and (6) of Lemma \ref{punto}, then
$$\rho^2 \int_{\theta_i}^{\theta'_i}
\mu(\theta)s^\perp(\theta)\wedge d[\mu(\theta)s^\perp(\theta)]=2(1-\rho^2) \mathcal{A}(B_{w_1}^{-w_{\sigma(1)}}),$$
and the equality (\ref{integral}) holds.

%%%%%%%%%%%%%%%%%%%%%%%%%%%%%%%%%%%%%%%%%%%%%%%%%%%%%%%%%%%%%%%%%%%%%%%%%%%%%%%%%%%%%%

\medskip

\textit{Step 3: If $\bar{v}$ verifies} (\ref{support}) \textit{and $v_i\prec\bar{v}\prec -v_{\sigma(i)}$, then there exist $\bar{v}', \bar{v}''\in S$ such that  both $\bar{v}'$ and $\bar{v}''$ verify} (\ref{support})\textit{, and   $v_i\prec\bar{v}'\prec \bar{v}\prec \bar{v}''\prec -v_{\sigma(i)}$.}

Let us prove the existence of $\bar{v}'$ (the existence of $\bar{v}''$ can be proved similarly). Only for simplicity, let us assume that $i=1$. Let us consider
$$\bar{v}_{1}=\bar{v},\hspace{1cm}\bar{v}_{j+1}=\bar{v}_j^*\hspace{1cm}P_{\bar{v}}=\{\bar{v}_{1}, \bar{v}_{2},\ldots, \bar{v}_n\}$$
$$\bar{w_j}=\frac1{2\rho}(\bar{v_j}+\bar{v}_{j+1})\hspace{1cm} \bar{w}_{j+1}=\bar{w}_j^* \hspace{1cm}P_{\bar{w}}=\{\bar{w}_{1}, \bar{w}_{2},\ldots, \bar{w}_n\}$$

Since $\bar{v}$ verifies (\ref{support}), then ((2) of Lemma \ref{elip}) $C_{\bar{v}}$ and $S$ are tangent at every point of $P_{\bar{v}}\cup P_{\bar{w}}$. Moreover:
\begin{enumerate}
	\item[(a)] ${\bar{v}}_j\wedge {\bar{v}}_{j+1}= v_i\wedge v_{i+1}$ (by (4) of Lemma \ref{punto} and Lemma \ref{ig}).
	\item[(b)] $\mathcal{A}(B_{v_i\bar{v}_i})$ is constant (by (a) and Lemma \ref{T}).
	\item[(c)] $\mathcal{A}(B_{v_i{\bar{v}_i}})=\mathcal{A}(B_{w_i\bar{w}_i})$ (by (3) of Lemma \ref{punto}).
\end{enumerate}
Using the above properties, the existence of $\bar{v}'$ of this Step 3 (such that $v_i\prec\bar{v}'\prec \bar{v}$) can be proved applying the arguments of Step 2 to $v_i$ and $\bar{v}$ (instead to $v_i$ and $-v_{\sigma(i)}$).

%******************************************** por aqu\'{\i}**********
%
%y entonces, basta tener en cuenta que, por apartado 4 Lemma \ref{punto} y Lemma \ref{ig}, ${\bar{v}}_i\wedge {\bar{v}}_{i+1}= v_i\wedge v_{i+1}$ y que, ademas, por Lemma \ref{T}, $A(B_{v_i\bar{v}_i})$ es constante para, teniendo en cuenta apartado (3) Lemma \ref{punto}, concluir que $A(B_{v_i{\bar{v}_i}})=A(B_{w_i\bar{w}_i})$, lo que garantiza la existencia de $\bar{v}'$.

\medskip

\textit{Step 4: The set of points of $S$ that verify} (\ref{support}) \textit{is dense on $S$.}

Let $x$, $y$ be a pair of points of $S$ that verify (\ref{support}).
Step 3 can be applied to $x$ and $y$ (instead of $v_i$ and $\bar{v}$) because $x$ and $y$ have the required properties: both verify property (\ref{support}) and the conditions for $x$ and $y$ equivalent to (a), (b), and (c) remind true (by the same reasons). Therefore, there exists $z\in S$ such that $z$ verifies (\ref{support}) and $x\prec z\prec y$.

As a consequence of Step 4, the statements of Lemma \ref{todos} are true for every $u\in S$. Particularly, because conditions similar to (a) and to (b) are verified for any pair of points of $S$ (instead of $v_i$ and $\bar{v}_i$), then $u\wedge u^*=v\wedge v^*$ and $\mathcal{A}(B_u^v)=\mathcal{A}(B_{u^*}^{v^*})$ for any $u,v\in S$ such that $u\prec v$.

%the following properties hold:
%$P_u$ has $n=2m+1$ vertices; the $C_u$ is tangent to $S$ at each one
%of them; the function $u\in S\to u\wedge u^*$ is
%constant; and $\mathcal{A}(B_{uv})=\mathcal{A}(B_{u^*v^*})$ for every $v\in S$ such that $u\prec v$.

\end{proof}

\section{Main result}\label{maintheo}
The last lemma describes some properties of a natural parametrization $s$ and the related parametrization $s^*$.
%%%%%%%%%%%%%%%%%%%% PROPORCIONAL ORTOGONAL Y DIFERENCIAL %%%%%%%%%%%%%%%%%%

\begin{lem}\label{prop}
Let $\rho \in M$ and  $s:[0,2\pi]\to S$ be a natural parametrization for $S$. If $X$ fulfils \normalfont(P-$\rho S$), then:

{\rm(i)} $s$ is continuously differentiable and there is a continuous function $p:[0,2\pi]\to\mathbb{R}_+$ such that
$s'(\theta)=p(\theta)s^\perp(\theta)$.

{\rm(ii)} $s^*$ is continuously differentiable and there is a continuous function $q:[0,2\pi]\to\mathbb{R}_+$ such that
$s^{*\prime}(\theta)=q(\theta)s^{*\perp}(\theta)$.
\end{lem}

%****************************************
%
%CREO QUE ESTE RESULTADO NECESITA LA HIP\'{O}TESIS $\rho\in M$, QUE NO ESTABA. LO DIGO PORQUE LA DEMOSTRACI\'{O}N SE BASA EN LA NOTA 2.40, P\'{A}GINA 100 DE LA TESIS, Y AH\'{I} SE NECESITA EL LEMA 2.37 DE LA TESIS ($u\wedge u^*$ es constante, Lemma \ref{todos} aqu\'{\i}).
%
%HAY QUE AMPLIAR LOS COMENTARIOS SOBRE ESTA DEMOSTRACI\'{O}N Y PONER LO EQUIVALENTE A LA NOTA 2.40 DE LA TESIS, CAMBIANDO LO NECESIARIO, ES DECIR, HACIENDO REFERENCIA A LOS RESULTADOS INCLUIDOS EN ESTE ART\'{I}CULO.
%
%*****************************************

\begin{proof}
	The following conditions holds: $X$ is smooth (by Lemma \ref{b}); the function $u\in S\to u\wedge u^*$ is constant (by Lemma \ref{todos}); and
		 $s$, $s^*$, and $s^{\perp}$ are continuous functions (see Section \ref{prelemma}). Besides, $s'(\theta)\wedge s^*(\theta) \neq 0$ as a consequence
of $s(\theta)\perp (1-2\rho^2)s(\theta)+s^*(\theta)$ for every  $\theta\in[0,2\pi]$ (by Lemma \ref{todos} and because $1-2\rho^2\neq 0$).
%	\begin{enumerate}
		%\item $X$ is smooth (Lemma \ref{b}),
%		\item $u\wedge u^*$ is constant $\forall u\in S$ (Lemma \ref{todos}),
%		\item $s$, $s^*$, and $s^{\perp}$ are continuous functions (see Section \ref{prelemma}),
%		\item \label{ine}$s'(\theta)\wedge s^*(\theta) \neq 0$.
	%\end{enumerate}
%	The inequality (\ref{ine}) is a consequence of  $s(\theta)\perp (1-2\rho^2)s(\theta)+s^*(\theta)$ for every  $\theta\in[0,2\pi]$ (Lemma \ref{todos} and $1-2\rho^2\neq 0$).
Therefore, the proof of the statement for the case	$\rho=\tfrac12$ (Lemma 2.8 in \cite{BY1}) can be rewritten for $\rho\in M$.
\end{proof}

%%%%%%%%%%%%%%%%%%%%%%%%% TEOREMA %%%%%%%%%%%%%%%%%%%%%%%%%
And finally, the main result is presented.
\begin{teo}\label{teo}
Given the set
$$
M=\left\{ \rho\in (0,1) / \, \rho=\sqrt{(1+\cos\tfrac{2k\pi}{2m+1})/2}:\; k=1,2,\ldots , m;\:m=1,2,\ldots \right\},
$$
a real normed space $X$ is an i.p.s. if and only if there exists $\rho \in M$ such that
$X$ fulfils
\begin{equation*}
u,v\in S,\:\inf_{t\in[0,1]}\|tu+(1-t)v\|=\rho \: \Rightarrow \:
\tfrac12u+\tfrac12v\in \rho S.\tag{P-$\rho S$}
\end{equation*}

\end{teo}

\begin{proof}
Let $X$ be an i.p.s. such that the scalar product of $u,v\in X$ is $(u|v)$. It is easy to see that for any $u,v\in S$, $u\prec v$, the convex function
$$
F(t)=\|(1-t)u+tv\|^2=1-2t+2t^2+2t(1-t)(u|v)
$$
attains its minimum at $t=\tfrac12$ when $(u|v)<1$. Thus, $X$ fulfils (P-$\rho S$) for every $\rho\in (0,1)$.

In order to prove the converse, let us fixed a natural parametrization $s:[0,2\pi]\to S$ for $S$. The following conditions holds:
	\begin{enumerate}
		\item 	$X$ is smooth (Lemma \ref{b}).
		\item	If $u,v\in S$ such that $u\prec v$, then 
	 $u \wedge u^*=v\wedge v^*$; $\mathcal{A}(B_u^v)=\mathcal{A}(B_{u^*}^{v^*})$; $u\perp(1-2\rho^2)u+u^*$; and $u^*\perp -u-(1-2\rho^2)u^*$ (Lemma \ref{todos}, property (\ref{support}) for $u$, and property (\ref{support2}) for $u^*$).
	\item %If $s:[0,2\pi]\to S$ is a natural parametrization for $S$,
There exist  some continuous functions $p, q:[0, 2\pi] \to\mathbb{R}_+ $  such that $s'(\theta)=p(\theta)s^{\perp}(\theta)$ and $s^{*\prime}(\theta)=q(\theta)s^{*\perp}(\theta)$ (Lemma \ref{prop}).

	\end{enumerate}
Using (1), (2), and (3), the proof of the statement for the case $\rho=\tfrac12$ (Theorem 3.1 in \cite{BY1}) can be rewritten for $\rho\in M$ with only very slight and not significant changes. For example, the (non restrictive) initial data $s(0)=(1,0)$ and $s^*(0)=(-\frac{1}{2}, \frac{\sqrt{3}}{2})$ considered for $\rho=\tfrac12$ would be replaced by $s(0)=(1,0)$ and $s^*(0)=(\cos \tfrac{2k\pi}{2m+1}, \sin \tfrac{2k\pi}{2m+1})$.
\end{proof}

\section{Conclusion and Open Problem}\label{conclusion}
We conjecture that a real normed space $X$ is an i.p.s. if and only if there exists $0<\rho<1$ such that $X$ fulfils property (P-$\rho S$). %Given the set
%$$M=\left\{ \rho\in \mathbb{R} / \, \rho=\sqrt{\left(1+\cos\frac{2k\pi}{2m+1}\right)/2}:\; k=1,2,\ldots , m;\:m=1,2,\ldots \right\},$$
If $2k<n$ and $n=3,4,...$, the case  $\rho\neq\sqrt{(1+\cos\tfrac{2k\pi}{n})/2}$ is proved in \cite{BY2}. The case $\rho=\sqrt{(1+\cos\tfrac{2k\pi}{n})/2}$ is proved in this paper when $n$ is odd (for $n=3$, $\rho=\tfrac12$, it was solved previously in \cite{BY1}), but it is left open  when $n$ is even. For this unsolved situation, the results of Section \ref{prelemma} and Section \ref{ellipsepolygon}, as well as some assertions  of Lemma \ref{punto} ((4), (5); also (1) and (2) considering $n=2m$) remain true. Besides, regarding (3) of Lemma \ref{punto}, it is easy to see that $P_v=P_{-v}$ and $P_v\cap P_w=\emptyset$ when $n$ is even. Nevertheless, the authors are not able to prove that the vertices of $P_v$ and $P_w$ split $B$ into $2n$ disjoint sectors of equal area, which would be the property equivalent to (6) of Lemma \ref{punto}.

\section*{Acknowledgement}
The second and third authors are very grateful to the first author for his work along more than 30 years in the University of Extremadura and, in general, for the example of his life. This paper is one of the last mathematical contributions of professor Carlos Ben\'{\i}tez, who died on March 7th, 2014.

\end{document}